\documentclass{amsart}
\pdfoutput=1
\usepackage{amsmath}
\usepackage{paralist}
  \usepackage{graphics} 
  \usepackage{epsfig} 
\usepackage{graphicx}  \usepackage{epstopdf}
 \usepackage[colorlinks=true]{hyperref}
\hypersetup{urlcolor=blue, citecolor=red}

\usepackage[noabbrev,capitalise]{cleveref}

\newtheorem{theorem}{Theorem}[section]

\newtheorem{lemma}[theorem]{Lemma}

\newtheorem{question}{Question}

\theoremstyle{definition}
\newtheorem{definition}[theorem]{Definition}
\newtheorem{remark}{Remark}

\newcommand{\N}{\mathbb{N}}

\title[On Periodic Orbits of Evolutionary Games] 
      {On Arbitrarily Long Periodic Orbits of Evolutionary Games on Graphs}

\author[Jeremias Epperlein and Vladim\'{i}r \v{S}v\'{i}gler]{}

\subjclass{Primary: 91A22, 05C57, 91A43; Secondary: 37N40.}
 \keywords{Evolutionary games on graphs, game theory, periodic orbit, deterministic imitation dynamics, discrete dynamical systems.}

 \email{jeremias.epperlein@gmail.com}
 \email{sviglerv@ntis.zcu.cz}

\thanks{$^*$ Corresponding author.}

\begin{document}
\maketitle

\centerline{\scshape Jeremias Epperlein$^*$}
\medskip
{\footnotesize
 \centerline{Center for Dynamics \& Institute for Analysis}
   \centerline{Dept. of Mathematics}
   \centerline{Technische Universit\"{a}t Dresden, 01062}
   \centerline{Dresden, Germany}
} 

\medskip

\centerline{\scshape Vladim\'{i}r \v{S}v\'{i}gler }
\medskip
{\footnotesize
 \centerline{Department of Mathematics and NTIS}
   \centerline{Faculty of Applied Sciences}
   \centerline{University of West Bohemia}
   \centerline{Univerzitn\'{i} 8, 30614}
   \centerline{Pilsen, Czech Republic}
}

\bigskip

\begin{abstract}
A periodic behavior is a well observed phenomena in biological and
economical systems. We show that evolutionary games on graphs with imitation dynamics can display periodic behavior for an arbitrary choice of game theoretical parameters describing social-dilemma games. We construct graphs and corresponding initial conditions whose trajectories are periodic with an arbitrary minimal period length. We also examine a periodic behavior of evolutionary games on graphs with the underlying graph being an acyclic (tree) graph. Astonishingly, even this acyclic structure allows for arbitrary long periodic behavior.
\end{abstract}

\section{Introduction}
\label{sec:int}
Evolutionary game theory on graphs in the spirit of Nowak and May \cite{nowak1992} studies the
evolution of social behavior in spatially structured populations.
In our setting, each vertex of a graph is assigned a strategy.
In every time step, each vertex plays a matrix game with its imminent neighbors.
The resulting game utilities together with the update order (certain vertices can remain rigid) and the update function result in the change of strategy to the subsequent time step.
Here, we focus on the case of synchronous update order and deterministic imitation dynamics - every vertex copies the strategy of the most successful neighbor including itself.
From a biological point of view, it is natural to
consider a stochastic update rule
leading mathematically to a Markov chain.
This is also the approach taken by most rigorous investigations
of such systems, see for example \cite{cox2013} and \cite{allen2014}.
Randomness can also be used to introduce mutations of the individuals into the model, see for example \cite{allen2014}.
However, questions about the dynamical
behavior of these models often become intractable because of
the stochastic nature of the system.
Different authors therefore also studied deterministic versions of the model,
see for example \cite{abramson2001, duran2005, kim2002,
  masuda2003, nowak1992, tomochi2004}. The results regarding this deterministic version, however, are almost all obtained by simulations.

One of the factors influencing dynamics heavily are the parameters of the underlying matrix game.
We consider a two strategy game (Cooperation, Defection) whose interpretation leads to a natural division of the parameter space into $4$ scenarios Prisoner's dilemma, Stag hunt, Hawk and dove, Full cooperation.
The equilibria of the replicator dynamics based on matrix games with such parameters are already well known, \cite{hofbauer1998}.

In \cite{nowak1992} it was shown by simulations that even on lattices these
dynamical systems can show very complicated behavior starting
from a very simple initial condition, see also Chapter 9 in \cite{nowak2006}.
Nowak and May show for example that the systems can exhibit cascading behavior.
Most of the time, such constructions work for very specific choices of
parameters.

Our main questions is thus: Can arbitrarily long periodic behavior happen for all parameter choices?
For example, the replicator dynamics with parameters of HD scenario tend to a stable mixed equilibrium but only pure equilibria are attractive for the other scenarios.
The spatial structure of the game must be thus thoroughly examined. We will answer the question positively by
explicitly constructing the graphs demonstrating the required behavior for each set of parameters.

In this paper, we focus on evolutionary games on a graph with periodic trajectories along which the strategy profiles (number of cooperators and defectors) change.
Periodic behavior with no change in the strategy profiles was observed for example in \cite{nowak2006} by Nowak and May in a structure they called a walker (spaceship in cellular automata terms).
Moreover, such a structure is automorphism invariant in a certain sense; for each time step, there exists a graph automorphism which keeps the moving walker in one place.
Such structures may be of interest for future research. We note, that our constructions introduce a periodic behavior of arbitrary length both in strategy vectors (distribution of strategies) and strategy profiles.

Evolutionary games on graphs  also form a very interesting class of cellular automata.
Cellular automata on a lattice can take into account
the relative spatial position of a neighbor. The dependence on
the neighbor on the left might differ from the dependence on the
neighbor on the right.
On an arbitrary graph this is only possible if the edges carry some
kind of label.
A cellular automaton in which the new state of a cell depends only on its
own state and the number of neighboring cells in each state is called totalistic.
Such cellular automata are naturally defined also on unlabeled graphs.
While evolutionary games as considered here are not totalistic,
they nevertheless are defined on unlabeled graphs in an obvious way.
See \cite{marr2009} for a discussion of cellular automata on graphs.

The paper is organized as follows. We introduce basic notation and
the dynamics of evolutionary games on a graph in \cref{sec:pre}. In
\cref{sec:con}, Theorem \ref{con-thm:1}, answering the main question of
this paper, is stated and proved. The proof is carried out for two
separate cases depending on the parameter scenario. Periodic behavior
of evolutionary games on a graph with the underlying graph being a
tree is examined in \cref{sec:tre}. We conclude our results in \cref{sec:ccl}.

\section{Preliminaries}
\label{sec:pre}
We are considering undirected connected graphs $\mathcal{G}$ as the
spatial structure of our game with the vertices $V$ being players. The
interactions between vertices are defined by a set of edges $E$ (no
edge means no direct interaction). The $m$-neighborhood of the vertex
$i$ (the set of all vertices having distance to $i$ exactly $m$) is
denoted by $N_m(i)$. We also define
\begin{align*}
N_{\leq m}(v) := \bigcup_{n=1}^m N_n(v) \cup \{v\}.
\end{align*}
The strategy set of the game is simply
$S = \{ \text{"Cooperate"},\text{"Defect"} \} = \{ C,D\} =
\{1,0\}$. The neighboring vertices play a matrix game where
the resulting utilities are defined by the matrix
\begin{align*}
\begin{array}{c|cc}
  & C & D \\ \hline
C & a & b \\
D & c & d
\end{array}
\end{align*}
and the utility function $u$. For example, if player A cooperates and player B defects, player A gets utility b and player B gets utility c.  In each time step, a certain subset of players is allowed to change their strategy based on the update order $\mathcal{T}$. Finally, the strategy update is defined by a function $\varphi$. A general framework of evolutionary games on graphs was developed in \cite{eg2}. An evolutionary game on a graph can be formally defined as follows.
\begin{definition}
An \textit{evolutionary game on a graph} is a quintuple $(\mathcal{G},\pi,u,\mathcal{T},\varphi)$, where
\begin{enumerate}[(i)]
\item $\mathcal{G} = (V,E)$ is a connected graph,
\item $\pi = (a,b,c,d)$ are game-theoretical parameters,
\item $u: S^V \to \mathbb{R}^V$ is a utility function,
\item $\mathcal{T}: \mathbb{N}_0 \to 2^V$ is an update order,
\item $\varphi: (\mathbb{N}_0)^2_\geq \times S^V \to S^V$ is a dynamical system.
\end{enumerate}
\end{definition}
The strategy vector (the state of the system) will be denoted $X = (x_1, \ldots, x_{| V|}) \in S^V$.
For a
strategy vector $X \in S^V$, the strategy of the vertex $v$ is $X_v$.
The utility of a player $v$ is given by $u_v(X)$.

Our main focus lies in social dilemma games and we are thus interested
in the game-theoretical parameters $\mathcal{P}$ describing such
games. In particular, it is more advantageous if the opponent
cooperates than if it defects for each player, i.e.
\begin{align*}
\min \{ a, c \} > \max \{b, d \} \, .
\end{align*}

This results into four possible scenarios: Prisoner's dilemma (PD):
$c>a>d>b$, Stag hunt (SH): $a>c>d>b$, Hawk and dove (HD): $c>a>b>d$
and Full cooperation (FC): $a>c>b>d$. Demanding the inequalities to be
strict only excludes sets of measure zero. This \textit{generic
  payoff assumption} is common in examining game-theoretical models
(see e.g.\,\cite{broomrychtar2013}). From now on, we consider the
\textit{mean utility function}

\begin{multline*}
u_i(X) = \frac{1}{\vert N_1(i) \vert} \left( a \sum_{j \in N_1(i)} X_i X_j + b \sum_{j \in N_1(i)} X_i(1 - X_j) + \right. \\
 \left. + c \sum_{j \in N_1(i)} (1- X_i) X_j + d \sum_{j \in N_1(i)} (1 - X_i)(1 - X_j) \right)
\end{multline*}
which is an averaged sum of the outcomes of the matrix games played
with direct neighbors. Without loss of generality, we can now assume
$a=1, d=0$ and thus the parameter regions can be depicted in the plane
(see Figure \ref{pre-fig:10}).

\begin{figure}[htbp]
\centering
\includegraphics[width = .65\textwidth]{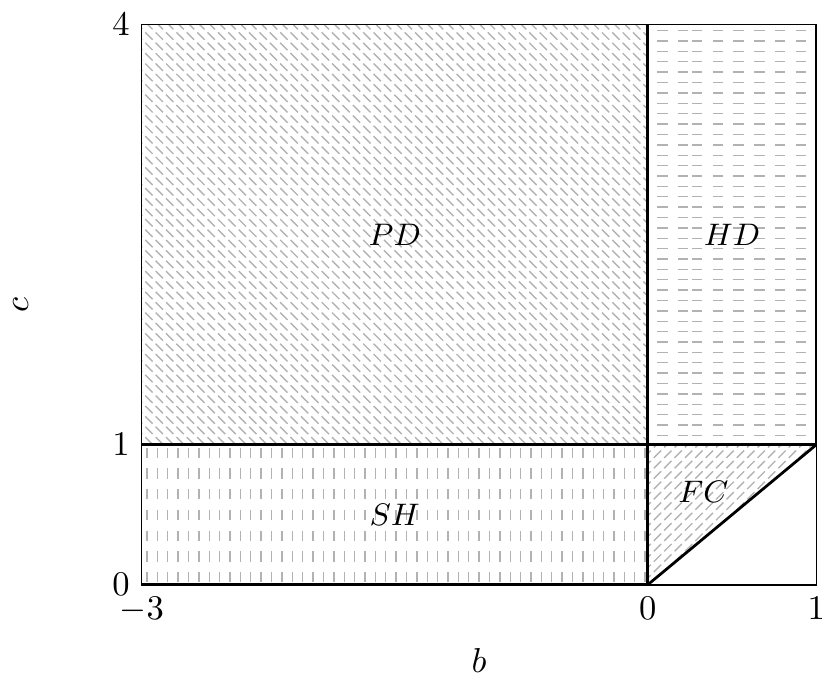}
\caption{Regions of admissible parameters $\mathcal{P}$
with normalization $a=1, d=0$.}
\label{pre-fig:10}
\end{figure}
This simplification can be done thanks to the averaging property of
the mean utility function (see \cite[Remark 8.]{eg1}). In Question \ref{pre-que:1}, Theorem \ref{con-thm:1} and Theorem \ref{tre-thm:1} we assume a
synchronous update order only, i.e.\ $\mathcal{T}(t)= V$ for all $t \in
\mathbb{N}_0$. In other words, all vertices are updated simultaneously
at every time step.  However, the definitions from Section
\ref{sec:pre} make sense for an arbitrary update order. Finally, the dynamical system $\varphi$ follows the deterministic imitation dynamics; namely, each vertex adopts the strategy of its most successful neighbor (including itself). Formally,
\begin{align*}
\varphi_i(t+1,t,X) = \left\{
\begin{array}{ll}
  X_{\text{max}} & i \in \mathcal{T}(t), \; \vert A_i(X) \vert = 1 \; \text{and} \; A_i(X) = \left\{ X_{\text{max}} \right\}  \, , \\
X_i  & \text{otherwise} \, ,
\end{array}
\right.
\end{align*}
where
\begin{align*}
A_i(x) = \left\{ X_k \, | \, k \in \text{argmax} \left\{ u_j(X) \, \vert \, j \in N_{1}(i) \cup \{ i \} \right\} \right\} \, .
\end{align*}
We refer to \cite{eg1} for further discussion on the utility function, update order and the dynamics.

Considering a deterministic dynamical system, the natural interest
lies in examining the existence and properties of \textit{fixed
  points} -- strategy vectors $X^* \in S^V $ for which $\varphi(t,0,X^*) = X^*$ for $t \in \mathbb{N}_0$. This topic was studied in \cite{eg2}. Another notion is the one of a \textit{periodic trajectory}, a periodic behavior of a game on a graph.
\begin{definition}
  Given an evolutionary game $\mathcal{E} = (\mathcal{G}, \pi, u,
  \mathcal{T},\varphi)$  on a graph and an initial
  state $X_0$, the sequence of strategy vectors
  $\mathcal{X} = (X(0),X(1), \ldots) \in \left[ S^V
  \right]^{\mathbb{N}_0}$ is called the \emph{trajectory} of $\mathcal{E} = (\mathcal{G}, \pi, u, \mathcal{T},\varphi)$ with initial state $X_0$ if
  for all $t \in \N_0$ we have
\begin{align*}
X(0) &= X_0 \, , \\
X(t+1) &= \varphi(t+1,t,X(t)).
\end{align*}
The trajectory is called \emph{periodic with period $p \in \N$} if $X(t+p) =
X(t)$ for $t \in \mathbb{N}_0$.
\end{definition}

The previous definition admits an arbitrary choice of the update order.
In general, two games $\mathcal{E}_1 = (\mathcal{G},\pi,u,\mathcal{T}_1,\varphi)$, $\mathcal{E}_2 = (\mathcal{G},\pi,u,\mathcal{T}_2,\varphi)$ with the same initial condition $X_0$ may not have the same trajectory.

Note that a vertex playing a certain strategy will keep its strategy
if surrounded by vertices playing the same strategy. Thus, it is
reasonable to define a cluster of cooperators and defectors and their
inner and boundary vertices. The inner (IC) and boundary (BC)
cooperators are defined by
\begin{align*}
V_{IC} &= \{ i \in V \; | \; X_i=1 \; \text{and} \; X_j = 1 \; \text{for all }  j \in N_1(i) \}\, , \\
V_{BC} &= \{ i \in V \; | \; X_i=1 \; \text{and there exists }  j \in N_1(i) \; \text{with }  X_j = 0 \} \,.
\end{align*}
Boundary (BD) and inner (ID) defectors are defined analogously.

The basic question we are answering in this paper can now be
formulated using the notation introduced in this section:
\begin{question}
\label{pre-que:1}
Given admissible parameters $\pi = (a,b,c,d) \in \mathcal{P}$, a
utility function $u$, an update order $\mathcal{T}$, dynamics
$\varphi$ and a number $p \in \mathbb{N}$, does there exist a
connected graph $\mathcal{G}$ and an initial state $X_0$ such that $\mathcal{X}$
is a periodic trajectory of the evolutionary game
$\mathcal{E} =(\mathcal{G},\pi,u,\mathcal{T},\varphi)$  on a graph with minimal period $p$?
\end{question}

\section{Existence of a periodic orbit of arbitrary length}
\label{sec:con}
In the following, graphs and subgraphs are denoted by
big calligraphic letters (e.g., $\mathcal{G}$), sets of
vertices are denoted by capital letters (e.g., $K$) and single
vertices are denoted by lower case letters (e.g., $v$).
\begin{theorem}
\label{con-thm:1}
Let $\pi = (a,b,c,d) \in \mathcal{P}$ be admissible parameters, $u$
the mean utility function, $\mathcal{T}$ the synchronous update order,
$\varphi$ deterministic imitation dynamics and $p \in
\mathbb{N}$. Then there exists a graph $\mathcal{G}$ and an initial
state $X_0$ such that $\mathcal{X} = (X(0), X(1), \ldots )$ is a
periodic trajectory of minimal length $p$ of the evolutionary game
$\mathcal{E}=(\mathcal{G},\pi,u,\mathcal{T},\varphi)$ on a graph with initial state $X_0$.
\end{theorem}
Theorem \ref{con-thm:1} formally answers Question \ref{pre-que:1}. The proof will be carried out for the cases $a>c$ (FC and SH scenario) and $c>a$ (HD and PD scenario) separately. We construct a connected graph, define an initial state and show, that the resulting trajectory is periodic with the required minimal period length $p$.

\subsection{Proof of Theorem \ref{con-thm:1} for FC and SH scenarios}
\label{ssec:ac}
\newcommand{\setsep}{\;|\;}

\subsubsection{ The graph and initial state}
The construction of our graph depends on $p$ and three parameters $q,r,s
\in \N$ which
we will choose later. See \cref{ca-fig:a-greater-c} for an illustration of
the graph structure.
Let $\mathcal{S}$ be the bipartite graph with classes $S_1$ and $S_2$
each having $s$ vertices.
Add a vertex $h_{\mathcal{S}}$ incident with all vertices in $S_1$ and a vertex $f_{\mathcal{S}}$ incident with exactly one vertex in $S_2$.

Now take $2p-1$ copies of the complete graph with $q$ vertices, denoted
by $\mathcal{K}_{-(p-1)}, \dots,\mathcal{K}_{p-1}$, and chain them together to form a ladder-like structure.
Add one vertex $g$ connected to all vertices in $\mathcal{K}_0$.
Denote the vertices in $\mathcal{K}_n$ by $\{k_{n,\ell} \setsep \ell =
1,\dots,q\}$ for $n = -(p-1),\dots,p-1$ such that $k_{n,\ell}$ and
$k_{n+1,\ell}$ are connected by an edge for $n = -(p-1), \ldots, p-2$
and $\ell = 1,\dots,q$.
Add $q\cdot r$ many copies of $\mathcal{S}$ and denote them by $\mathcal{S}_{\ell,m}$ for $\ell = 1,\dots,q$ and $m = 1,\dots,r$.
Connect $f_{\mathcal{S}_{\ell,m}}$ to all vertices in $\{k_{n,\ell} \setsep n = -(p-1), \dots, p-1\}$ for $m = 1,\dots,r$ and $\ell = 1,\dots, q$.
We denote the graph thus obtained $\mathcal{G}$.

Set $H:=\{h_{\mathcal{S}_{\ell,m}} \setsep \ell = 1,\dots,q, \; m
= 1,\dots, r\}$. Let $I$ be the set of all neighbors of vertices in
$H$, let $J$ be the set of all neighbors of vertices in $I$ which are not already in $H$, and
set $F:=\{f_{\mathcal{S}_{\ell,m}} \setsep \ell = 1,\dots,q, \; m = 1,\dots, r\}$.
Finally set $K_n:=\{k_{n,\ell} \setsep \ell = 1,\dots,q\} \cup
\{k_{-n,\ell} \setsep \ell = 1,\dots,q\}$ for $n = 0,\dots,p-1$
and $K:=K_0 \cup \dots \cup K_{p-1}$.

Let $X_0$ be the state in which all vertices in $H \cup I \cup K_0 \cup \{g\}$ are cooperating
and all other vertices are defecting, that is,
\begin{align*}
  (X_0)_v &:=1, && v \in I \cup H \cup K_0 \cup \{g\} \, ,\\
  (X_0)_v &:=0, && v \in F \cup J \cup (K_1 \cup \dots \cup K_{p-1}) \, .
\end{align*}

\begin{figure}[htbp]
	\centering
	\includegraphics[width = 0.8\textwidth]{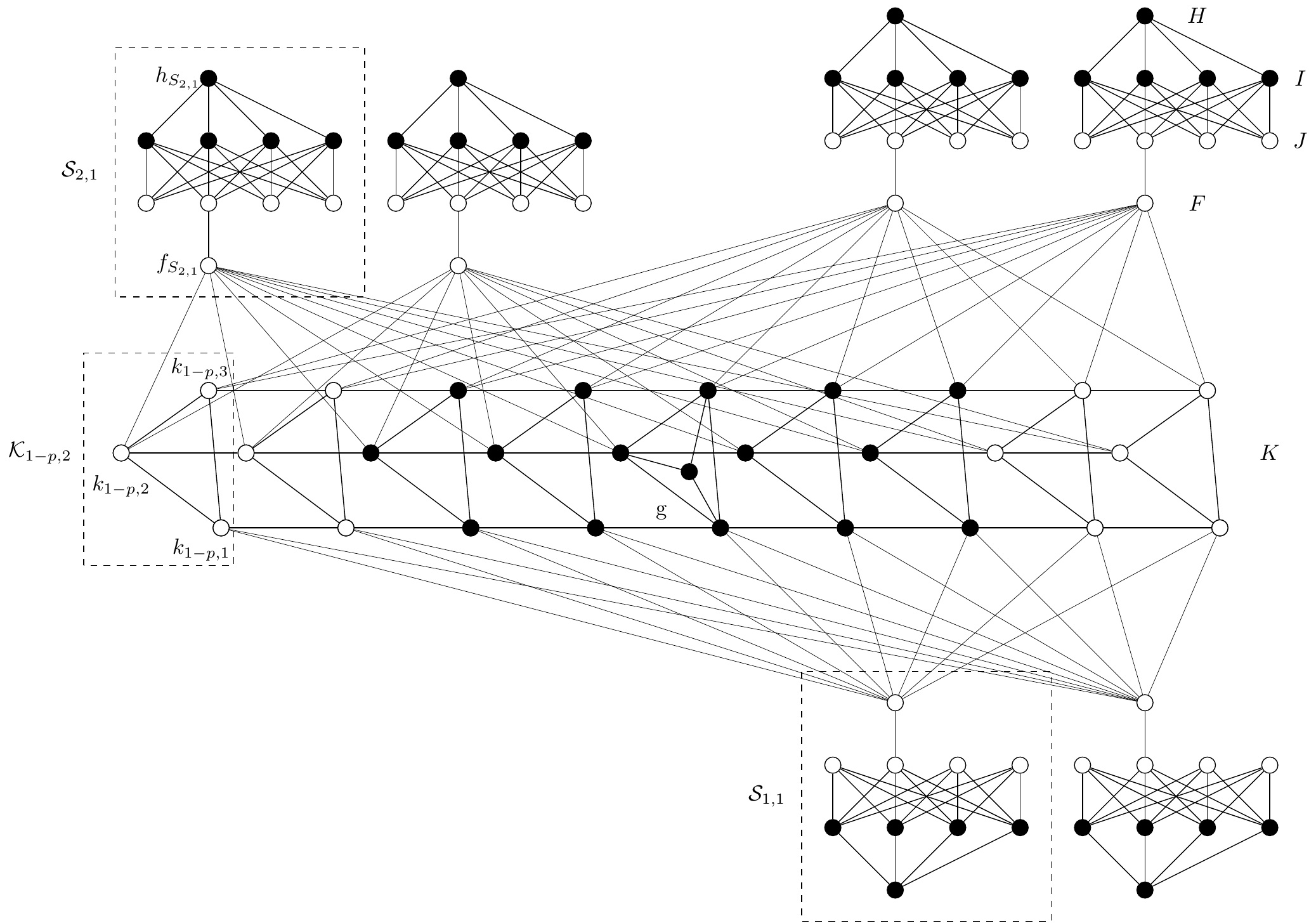}
	\caption{Example of the graph $\mathcal{G}$ with parameters $p=5$, $q=3$,
          $r=2$ and $s=4$. Cooperators are depicted by full circles. }\label{ca-fig:a-greater-c}
\end{figure}

\subsubsection{Dynamics}
Let $\mathcal{X} = (X(0),X(1),\dots)$ be the trajectory of the
evolutionary
game with parameters $(a,b,c,d)$,
synchronous update order, mean utility and imitation dynamics
on the graph $\mathcal{G}$ constructed above with initial state $X_0$.
Let $u_v(t):=u(X_v(t))$ be the utility of the vertex $v$ at time
$t$.
We will show that for suitable parameters $q,r,s$
the dynamics with initial value $X_0$ is
the following. All vertices not in $K$ do not change their strategy
and cooperation spreads along the ladder for $p$ time steps.
After $p$ time steps, we reach again the initial state $X_0$
since all vertices in $K \setminus K_0$ switch back to defection.
More precisely, for $t = 0,\dots,p-1$ we have
\begin{align*}
  X_v(t) &= X_v(0) \, ,  && v \in F \cup H \cup I \cup J \cup K_0 \cup \{g\} \, , \\
  X_v(t) &= 1 \, ,  && v \in K_{0} \cup \dots \cup K_{t} \, ,  \\
  X_v(t) &= 0 \, ,  && v \in K_{t+1} \cup \dots \cup K_{p-1} \,
\end{align*}
and $X(p)=X(0)$.
We first ensure that all vertices not in $K$ do not change their strategy.
The vertices in $H$ have utility $a$, the highest achievable one, and therefore
all vertices in $I \cup H$ always cooperate.
By the same argument, the vertex $g$ and its neighbors $K_0$ stay
cooperators at all time steps.
We want the vertices in $J$ to stay defectors by never imitating the strategy of their
neighbors in $I$, hence we want
\begin{align}  \label{eq:ac-1}
  u_j(t) &> u_i(t) \, , &t = 0,\dots,p-1, \; i \in I, \; j \in J \, .
\end{align}
Every vertex in $F$ should stay defecting. This is ensured if
\begin{align}\label{eq:ac-2}
  u_j(t) &> u_{k_n}(t) \, , &t = 0,\dots,p-1, \; j \in J, \; n = 0, \dots t, \; k_n \in K_n \, .
\end{align}
We now want cooperation to spread along the ladder for $t = 0,\dots,p-2$. At time $t$,
the vertices in $K_{t+1}$ should copy the strategy from the vertices in $K_t$ and all
vertices in $K_n$ with $n = 0,\dots, t$ should keep cooperating.
Defecting vertices without cooperating neighbors always have a lower utility
than defecting vertices with cooperating neighbors, hence the later
never imitate the former.

It is therefore sufficient to have
\begin{align}\label{eq:ac-3}
  u_{k_n}(t) &> u_f(t) \, , & t = 0, \dots, p-2, \; n = 0, \dots, t, \; k_n \in K_n, \; f \in F \, ,\\
\label{eq:ac-4}
u_{k_t}(t) &> u_{k_{t+1}}(t) \, , & t = 0, \dots, p-2, \; k_t \in K_t, \; k_{t+1} \in K_{t+1} \, .
\end{align}
In the time step from $p-1$ to $p$, we want the big reset to occur.
The utility of the vertices in $F$ should be greater than the utilities of all vertices
in $K$, in other words,
\begin{align}\label{eq:ac-5}
  u_{f}(p-1) &> u_k(p-1)\, ,&  f \in F, \; k \in K \, .
\end{align}

\subsubsection{Bounds for the utilities and the resulting inequalities}
We now give bounds for the utilities involved in the inequalities \eqref{eq:ac-1}
to \eqref{eq:ac-5}.
\begin{align*}
 u_j(t) &> \frac{s c + d}{s+ 1} \, , && t = 0,\dots, p-1, \; j \in J \, , \\
 u_i(t) &= \frac{a+s b}{s + 1} \, , && t = 0,\dots, p-1, \; i \in I \, , \\
 u_{f}(t) &< \frac{(2p-3)c+3d}{2p} \, , && t = 0,\dots, p-2, \; f \in F \, , \\
 u_{f}(p-1) &= \frac{ (2p-1)c + d}{2p} \, , && f \in F \, , \\
  u_{k_n}(t) &< \frac{ (q+2) a + r b}{q+2+ r} \, ,
  && t = 0,\dots, p-1, \;
    n = 0,\dots,t, \;
    k_n \in K_n \, , \\
  u_{k_n}(t) & > \frac{ q a + (r+2) b}{q+2+ r} \, ,
  && t = 0,\dots, p-2, \;
    n =  0,\dots,t, \;
    k_n \in K_n \, , \\
  u_{k_{t+1}}(t) &<\frac{ c + (q+r-1) d}{q + r}\, ,
  && t = 0,\dots, p-2, \;
    k_{t+1} \in K_{t+1} \, , \\
  u_{k}(p-1) & < \frac{ (q+2) a + r b}{q+r+2} \, ,
  && k \in K \, .
\end{align*}

A set of inequalities sufficient for \eqref{eq:ac-1} to \eqref{eq:ac-5} to hold is therefore given by
\begin{align}
  \frac{s c+d}{s+1} &> \frac{a+s b}{s+1} \, , \label{eq:ac-6}\\
  \frac{s c+d}{s+1} &> \frac{(q+2)a +r b}{q+r+2} \, ,\label{eq:ac-7} \\
  \frac{q a+(r+2) b}{q+r+2} &> \frac{(2p-3)c+3d}{2p} \, ,\label{eq:ac-8} \\
  \frac{q a+(r+2) b}{q+r+2} &> \frac{c+(q+r-1)d}{q+r} \, , \label{eq:ac-9}\\
  \frac{(2p-1) c+d}{2p} &> \frac{(q+2) a+r b}{q+r+2} \, . \label{eq:ac-10}
\end{align}

\subsubsection{Choosing parameters}

We start by choosing $r$ and $q$ in order to satisfy the inequalities \eqref{eq:ac-8} -- \eqref{eq:ac-10}.
Since $c>d$ we also have $ \frac{(2p-1) c+d}{2p} > \frac{(2p-3)c+3d}{2p}$ and
$ \frac{(2p-1) c+d}{2p} - \frac{(2p-3)c+3d}{2p} = \frac{c-d}{p}>0$.
Choose $m$ large enough such that $\frac{6(a-b)}{m+2} < \frac{c-d}{p}$ and $\frac{c+(m-1)d}{m}<\frac{(2p-3)c+3d}{2p}$.
We can then find $r \in \{1,\dots,m-1\}$ and set $q:=m-r$ such that
\begin{align*}
  \frac{(2p-1)c+d}{2p} > \frac{(q+2)a+rb}{m+2} > \frac{q a+(r+2)b}{m+2} > \frac{(2p-3)c+3d}{2p} \, .
\end{align*}
This directly implies  \eqref{eq:ac-8} and \eqref{eq:ac-10}.
The inequality  \eqref{eq:ac-9} follows from
\begin{align*}
  \frac{(2p-3)c+3d}{2p} > \frac{c+(q+r-1)d}{q+r} \, .
\end{align*}
Finally choose $s$ large enough such that \eqref{eq:ac-6} is fulfilled and such that
\begin{align*}
  \frac{s c+d}{s+1} > \frac{(2p-1)c+d}{2p} \, ,
\end{align*}
which implies  \eqref{eq:ac-7} by \eqref{eq:ac-10}. \qed

\subsection{Proof Theorem \ref{con-thm:1} for HD and PD scenarios}
\label{ssec:ca}
\subsubsection{The graph and initial state}
Let us define a graph $\mathcal{G}$ depending on $p$ and four
parameters $o,q,r,s \in \N$ as depicted in \cref{ca-fig:05}. We start
with $p$ complete graphs $\mathcal{K}_1, \ldots, \mathcal{K}_p$ on $o$
vertices. Again, the subgraphs $\mathcal{K}_n$ for $n = 1, \ldots, p$
are connected in series to form a ladder-like structure. There is a
subgraph $\mathcal{K}_{p+1}$ which is a complement of a complete graph
on $o$ vertices (isolated vertices) connected to the ladder in the
same manner. The vertices of $\mathcal{K}_n$ are denoted $k_{n,m}$ for
$n=1, \ldots, p+1$ and $m = 1, \ldots, o$, forming the sets $K_n$. Every vertex in the
interior of the ladder (the vertices in ${K}_2$ to ${K}_p$) is
connected to a vertex $g_R$. Additionally, the vertex $g_R$ has $q+1$
other neighbors. It has $q$ neighboring
vertices of degree one forming the set $H$ and a neighbor which we
call $g_D$. The vertex $g_D$ has $r$ neighboring vertices of degree
one forming the set $I$ and $s$ neighboring vertices of degree two forming
the set $J$. Each vertex in $J$ is connected to a vertex $g_C$.

Let $X_0$ be the initial state defined by
\begin{align*}
(X_0)_v &= 1 \, , && v \in K_1 \cup J \cup \{ g_C \} \, , \\
(X_0)_v &= 0 \, , && v \in \bigcup_{n=1}^{p+1} K_n \cup H \cup I \cup \{ g_D, g_R \} \, .
\end{align*}

\subsubsection{Dynamics}
Let $\mathcal{X} = (X(0),X(1),\dots)$ be the trajectory of the
evolutionary
game with parameters $(a,b,c,d)$,
synchronous update order, mean utility and imitation dynamics
on the graph $\mathcal{G}$ constructed above with initial state $X_0$.
We will show that for suitable parameters $o, q,r,s$
the dynamics with initial value $X_0$ is
the following.
Cooperation spreads along the ladder of vertices in
$\mathcal{K}_n$ to $\mathcal{K}_p$ and at time
$t=p-1$ the strategy of all vertices of
$\mathcal{K}_2$ to $\mathcal{K}_p$ is reset to defection. Formally
\begin{align}
X_v(t) &= 1\, , && v \in \bigcup_{n=1}^{t+1} {K}_n \, , \label{eq:c-greater-a-dyn-a}\\
X_v(t) &= X_v(t-1)\, , && \text{otherwise} \label{eq:c-greater-a-dyn-b}\, ,
\end{align}
for $t = 1, \ldots, p-1$ and $X(t+p) = X(t)$ for $t \in
\mathbb{N}_0$. See again \cref{ca-fig:05} for an illustration.
\begin{figure}[htbp]
\centering
\includegraphics[width = 0.9\textwidth]{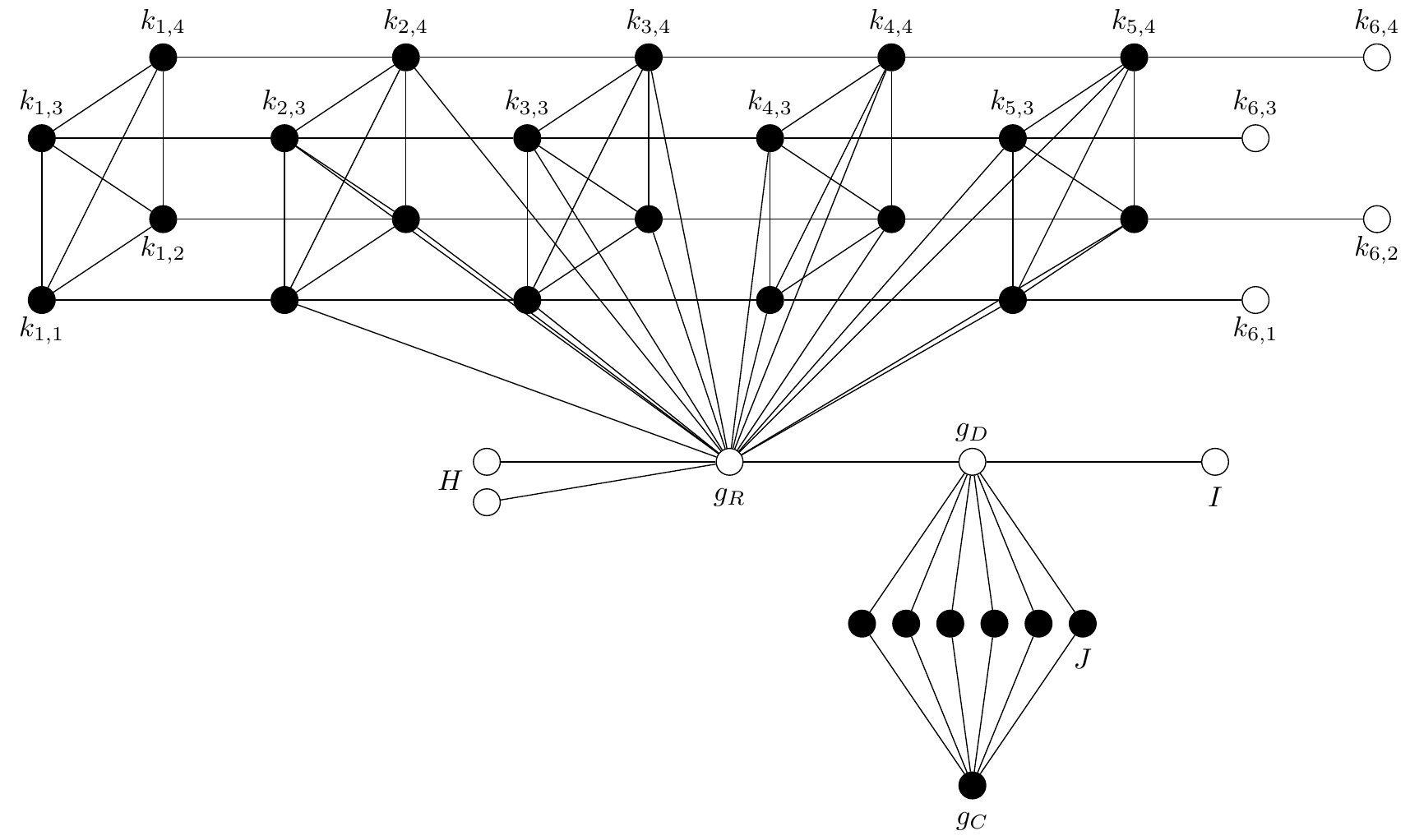}
\caption{Example of the graph $\mathcal{G}$ with parameters $p=5,o=4,s=6,r=1,q=2$ with strategy vector $X(4)$. Cooperators are depicted by full circles. Note, that this graph exhibits periodic behavior as described in Section \ref{ssec:ca} for $(a,b,c,d)=(1,0.45,1.24,0)$.}
\label{ca-fig:05}
\end{figure}

The following conditions must be satisfied in order for $\mathcal{X}$
to fulfill \eqref{eq:c-greater-a-dyn-a} and \eqref{eq:c-greater-a-dyn-b}.
\begin{itemize}
\item
  The vertices $g_R, g_C, g_D$ and all vertices in $H,I,J$ keep
  their strategy. The defector $g_D$ must prevent the vertex $g_R$
  from changing its strategy, must not change its own strategy and
  must not change the strategy of the cooperators in $J$. This is guaranteed by satisfying the inequalities
\begin{align}
\frac{(o+1)a+b}{o+2} < \frac{s c + (r+1) d}{s+r+1} < a \, ,
\label{ca-eq:60}
\end{align}
where $a$ is the utility of the vertex $g_C$ and the fraction on the
left hand side of \eqref{ca-eq:60} is an upper bound for the utilities
of the cooperating neighbors of $g_D$ and $g_R$.

\item
For the cooperation to spread at time $t$, $t=0, \ldots, p-2$, the
boundary cooperators in ${K}_{t+1}$ must have greater utility
than the boundary defectors in ${K}_{t+2}$, that is,
\begin{align}
\min \left\{ \frac{(o-1)a+b}{o}, \frac{o a + 2b}{o+2} \right\} > \frac{c + (o+1)d}{o+2} \, .
\label{ca-eq:10}
\end{align}
Here the first term on the left
is the utility of the cooperators in ${K}_1$ at $t=0$ and the second term is the utility of boundary cooperators in subsequent time steps.

\item The vertex $g_R$ must not be stronger than the cooperators in
  ${K}_m$ for $t =0, \ldots, p-2$ and $1 \leq m \leq t+1$ for
  the cooperation to be able to spread, hence
\begin{align}
\frac{oa + 2b}{o+2}>\frac{(p-2)oc+(o+q+1)d}{(p-1)o+q+1} \, .
\label{ca-eq:20}
\end{align}
Simultaneously, the defecting vertex $g_R$ must be able to change the
strategy of all neighboring cooperators to defection at time $p-1$, thus
\begin{align}
a <\frac{(p-1)oc+(q+1)d}{(p-1)o +q +1} \, .
\label{ca-eq:30}
\end{align}
\end{itemize}

\subsubsection{Choosing parameters}
We now show that there exists a choice of parameters $o,q,r,s \in \mathbb{N}$ such that the inequalities \eqref{ca-eq:60} -- \eqref{ca-eq:30} are satisfied.

Without loss of generality, we assume $a=1, d=0$ (see \cite{eg1},
Remark 8.). Since the denominators in \eqref{ca-eq:20},
\eqref{ca-eq:30} are positive, we can multiply both sides by the
product of the denominators and express $q$ in the terms of $o$. The inequality \eqref{ca-eq:20} gives
\begin{align}
q > \frac{o^2\left( (1-p)(1-c)-c \right)-o\left( 2(p-1)b - 2(p-2)c +1\right)-2b}{o+2b}
\label{ca-eq:40}
\end{align}
and the inequality \eqref{ca-eq:30} gives
\begin{align}
q < o(p-1)(c-1)-1 \, .
\label{ca-eq:50}
\end{align}
If we depict both inequalities in the first quadrant of the $o$-$q$
plane, the inequality \eqref{ca-eq:40} is satisfied above the line
given by the function on the right hand side. The function on the right
hand side asymptotically approaches the line with slope
\begin{align*}
\sigma_1 = (1-p)(1-c)-c \, .
\end{align*}
The inequality \eqref{ca-eq:50} is satisfied below the line with positive slope
\begin{align*}
\sigma_2 = (p-1)(c-1) > 0 \, .
\end{align*}
The difference of the slopes $\sigma_2-\sigma_1 =c>0$ is always
positive and  we are therefore able to find $o,q \in \mathbb{N}$ such
that the inequalities \eqref{ca-eq:20} and \eqref{ca-eq:30} are satisfied, see Figure \ref{ca-fig:10}. Furthermore, we can choose $o,q$ arbitrarily big.
\begin{figure}
\centering
\includegraphics[width=\textwidth]{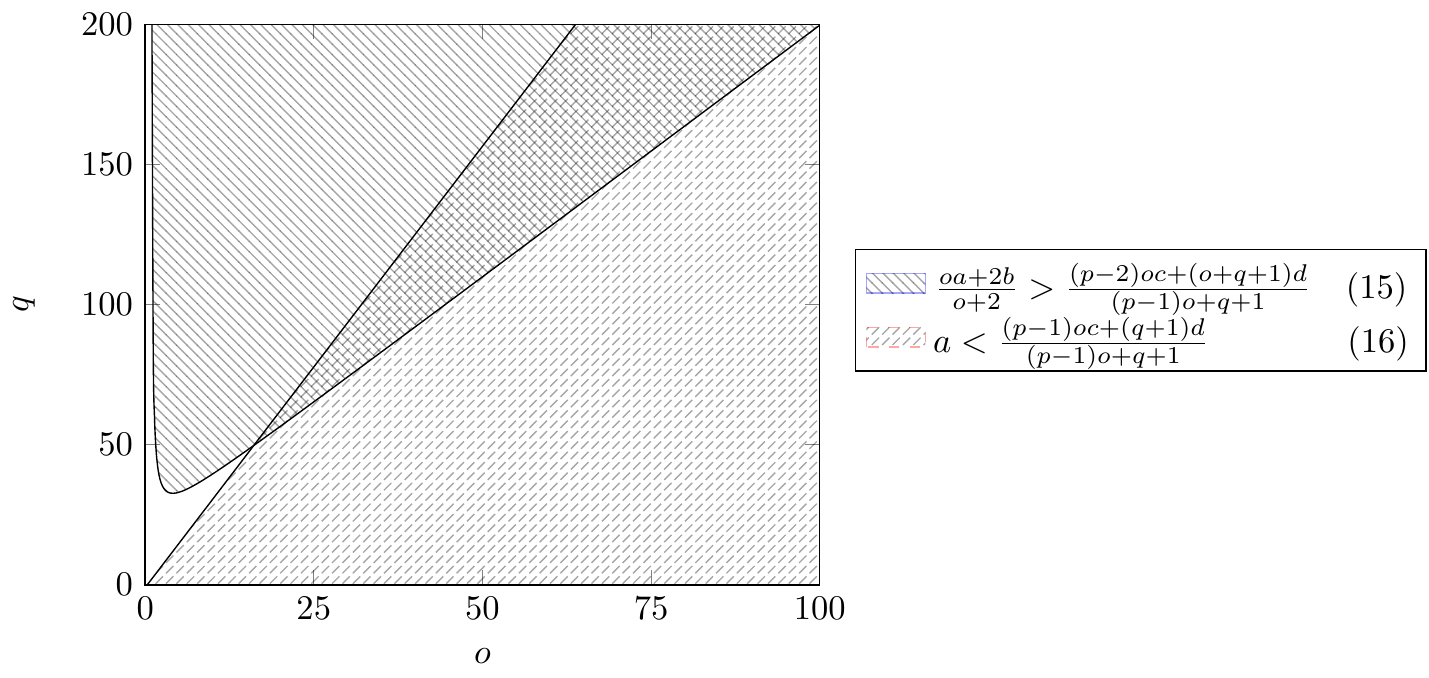}
\caption{Regions of parameters $o,q$ satisfying the inequalities \eqref{ca-eq:20}and \eqref{ca-eq:30}. The regions are depicted for  $(a,b,c,d)=(1,-0.45,1.35,0)$ and $p=10$.}
\label{ca-fig:10}
\end{figure}
Since $a>d$ holds, the number $o$ can be chosen great enough such that \eqref{ca-eq:10} is satisfied.

Since $c>a>d$ holds, we can find integers $r,s \in \mathbb{N}$ (possibly very big ones) such that \eqref{ca-eq:60} is satisfied (implicitly using the density of rational numbers and the fact that our parameters are generic).

Thus, we are able to find parameters $o,q,r,s \in \mathbb{N}$ such
that the equations \eqref{ca-eq:60} -- \eqref{ca-eq:30} are satisfied
and subsequently, $\mathcal{X}$ is a periodic trajectory of
$\mathcal{E} = (\mathcal{G},\pi,u,\mathcal{T},\varphi)$ with initial
state $X_0$ having minimal period length $p$. \qed

\section{Periodic orbits on an acyclic graph}
\label{sec:tre}
Interestingly, periodic behavior of an evolutionary game on a
graph can be observed even in the case when the underlying graph is a
tree. The absence of cycles demands a new view on the periodic
dynamics since the information (strategy change) can spread
only gradually through the graph; for example there is no way of
"resetting" vertex strategies. Nevertheless, for specific parameter
regions arbitrary long periodic behavior can occur.
\begin{theorem}
\label{tre-thm:1}
Let $\pi = (a,b,c,d) \in \mathcal{P}$ be admissible parameters
satisfying the conditions of the HD scenario, $c>a>b>d$, $u$ the mean
utility function, $\mathcal{T}$ the synchronous update order,
$\varphi$ deterministic imitation dynamics and $p_0 \in
\mathbb{N}$. There exists an acyclic graph $\mathcal{G}$, a number
$p \in \mathbb{N}_0$ such that $p \geq p_0$ and an initial state $X_0$
such that $\mathcal{X} = (X(0), X(1), \ldots )$ is a periodic
trajectory of minimal length $p$ of the evolutionary game
$\mathcal{E}=(\mathcal{G},\pi,u,\mathcal{T},\varphi)$ on a graph with
initial state $X_0$.
\end{theorem}

\subsection{Proof of \cref{tre-thm:1}}.
\label{tre-sec:proof}
\subsubsection{The graph and initial state}
Let us define a graph $\mathcal{G}$ whose structure is dependent on
two parameters $q,r$. The graph $\mathcal{G}$ is a rooted $r$-nary
tree such that
\begin{itemize}
\item the root $h_0$ has only one child $h_1$,
\item every vertex in level $1$ to $q-2$ has exactly $r$ children,
\item exactly $r^2$ vertices in level $q-1$ with pairwise different
  predecessors at level $3$ are leaves,
\item every other vertex in level $q-1$ has $r$ children which are leaves.
\end{itemize}
See \cref{tre-fig:05} for an illustration.

For the sake of simplicity, we focus only on one branch of the tree
$\mathcal{G}$ rooted in a fixed vertex $h_3$ at level $3$. The
vertices in the other branches follow the same dynamics by symmetry
reasons (the initial state and the graph $\mathcal{G}$ are invariant with respect to an automorphism exchanging the whole branches rooted at level $3$).
The descendant of $h_1$ at level $q-1$ in the fixed branch which is a
leaf is denoted by $h_{q-1}$. The
vertices in a path from $h_1$ to $h_{q-1}$ will be denoted
by $h_1, h_2, \ldots, h_{q-1}$ in an increasing manner. The vertices in
$\{ h_1, \ldots, h_{q-1}\} = H$ are called \textit{special
  vertices}. The set of all descendants of $h_m$ for
$m = 3, \ldots, q-2$ which are not in $H$  will be denoted by $I_m$. Vertices in
$I:= \bigcup_{m=3}^{q-2} I_m$ are called \textit{ordinary vertices}.

Let the initial condition $X_0$ be such that every vertex in levels
$0, \ldots, q-2$ is cooperating and every other vertex is defecting,
that is,
\begin{align*}
(X_0)_v &= 1 \, , && v \in N_{\leq q-2}(h_0) \, , \\
(X_0)_v &= 0 \, , && \text{otherwise} \, .
\end{align*}

See Figure \ref{tre-fig:05} for illustration of the graph construction and initial condition.

\begin{figure}
\includegraphics[width=\textwidth]{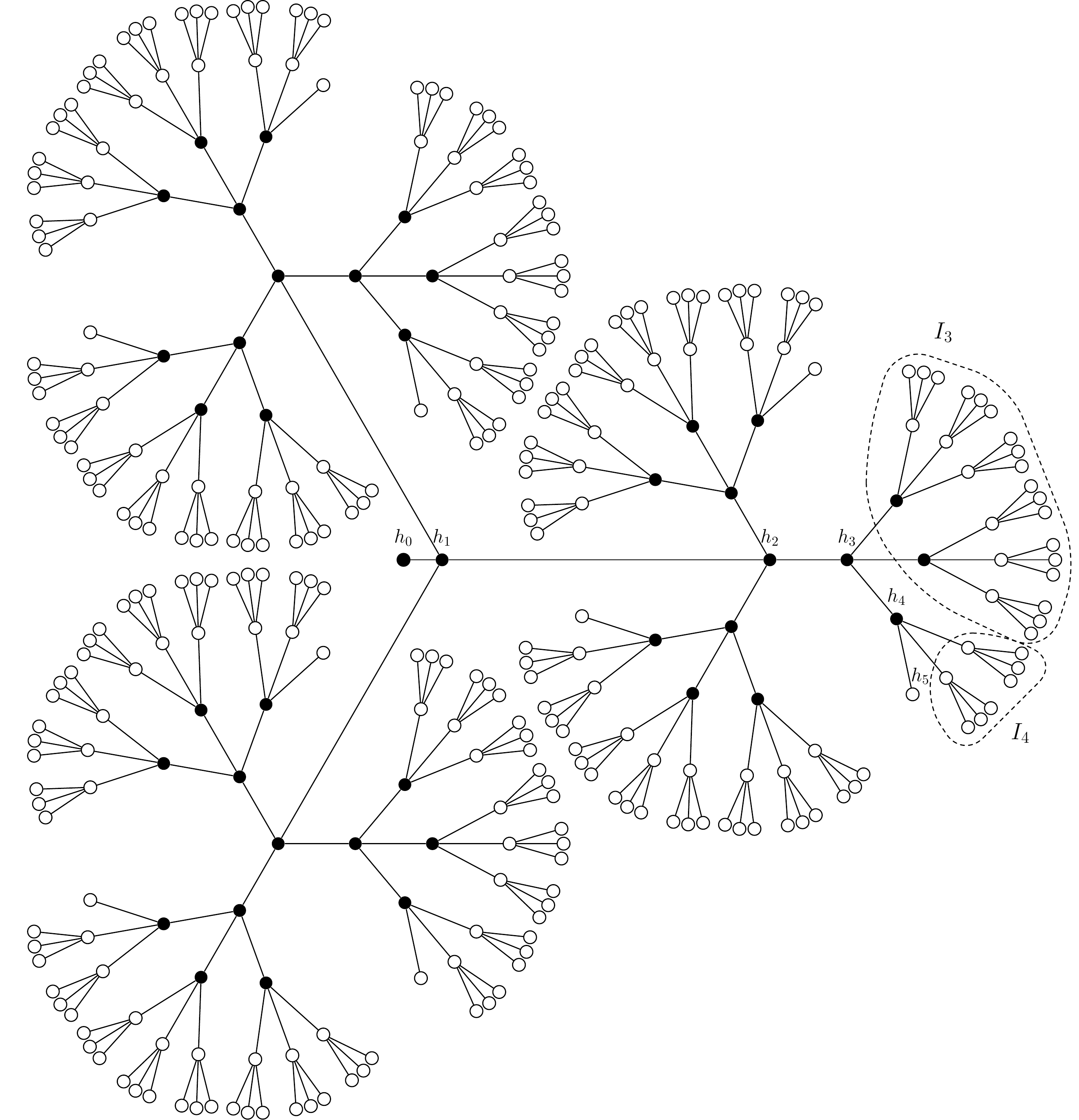}
\caption{Example of the graph constructed in the proof of Theorem
  \ref{tre-thm:1} with an initial condition. The cooperators are
  depicted by filled black circles, defectors by white ones. The parameters are $r=3, q=6$}
\label{tre-fig:05}
\end{figure}

\subsubsection{The dynamics}

The dynamics of the system in one period can be divided into three qualitatively different phases.
There are three important events that occur during one period.  At
time $t=0$, all vertices at level at most $q-2$ cooperate and all
vertices at the levels $q-1$ and $q$ defect.  From here, defection is
spreading along the special vertices to the root and outward towards
the boundary along the ordinary vertices. We call this phase the
\emph{shrinking phase}. At time step $t=q-3$, the only special vertices
which cooperate are those at level $0$ and $1$. There are however a
few clusters of cooperating ordinary vertices left in the higher
levels.  Starting from the root, the central cooperating cluster is
growing again and at time $t=(q-3)+(q-5)$ there is only this central
cluster of cooperators left which encompasses all vertices at level at
most $q-4$.  Two time steps later, at $t=2(q-3)$, this cluster
encompasses all vertices at level at most $q-2$ and we are back at the
initial state.
Please refer to the example in Section \ref{tre-sec:example} and
Figures \ref{tre-fig:ex0}--\ref{tre-fig:ex5} for an illustration of
the dynamics.

The local dynamics is essentially governed by the following two
lemmas.

\begin{lemma}\label{lem:claim1}
	Consider parameters $(a,b,c,d) \in \mathcal{P}$ such that
	\begin{align}
	\label{tre-eq:10}
	\frac{a+rb}{r+1} > \frac{c+rd}{r+1} \, .
	\end{align}
	Let $i$ be a vertex which is a boundary cooperator at time $t$.
	If $i$ is connected to one cooperator and $r$ boundary defectors, whose
	defecting neighbors have utility lower that $\frac{a+rb}{r+1}$
        and whose only cooperating neighbor is $i$, then
	$i$ and all of its defecting neighbors will cooperate in the
	next time step.
	\begin{figure}[htbp]
		\centering
		\includegraphics[width=0.8\textwidth]{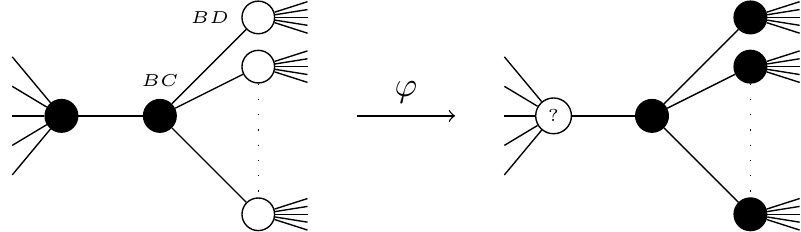}
		\caption{Illustration of the local situation in
                  \cref{lem:claim1}. Generally, nothing can be stated
                  about the behavior of the
			cooperating neighbor on the left.}
		\label{tre-fig:10}
	\end{figure}
\end{lemma}
\begin{proof}
	The defecting neighbors of $i$ have utility $\frac{c+rd}{r+1}$, and
	their defecting neighbors have utility smaller than $\frac{a+rb}{r+1}$. Both of these
	quantities are lower then the utility of $i$, which is $\frac{a+rb}{r+1}$.
\end{proof}

See Figure \ref{tre-fig:10} for an illustration of Lemma \ref{lem:claim1}.

\begin{lemma}\label{lem:claim2}
	Consider parameters $(a,b,c,d) \in \mathcal{P}$ such that
	\begin{align}
	\label{tre-eq:20}
	a < \frac{rc+d}{r+1} \, .
	\end{align}
	Let $i$ be a vertex which is a boundary defector
	at time $t$.
	If $i$ is connected to one defector and $r$ boundary cooperators,
	then $i$ and all its neighbors will defect in the next time step.
	\begin{figure}[htbp]
		\centering
		\includegraphics[width=0.8\textwidth]{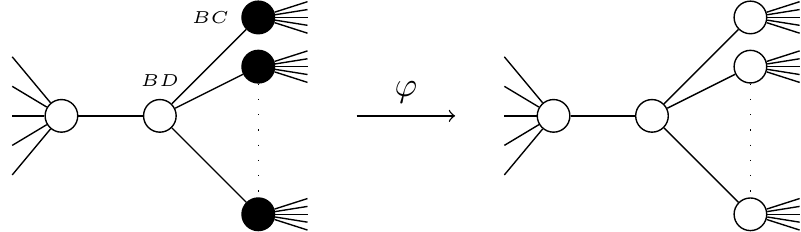}
		\caption{Illustration of the local situation in \cref{lem:claim2}.}
		\label{tre-fig:20}
	\end{figure}
\end{lemma}
\begin{proof}
	At time $t$, the vertex $i$ has utility $\frac{rc+d}{r+1}$ which
	is larger then $a$, the largest utility that a cooperator can achieve.
\end{proof}

See Figure \ref{tre-fig:20} for an illustration of Lemma \ref{lem:claim2}.

Let $\mathcal{X} = (X(0), X(1), \dots )$ be the trajectory
of the evolutionary game on a graph described
above with initial state $X_0$. We start with some simple observations.

\begin{lemma}
\label{lem:claim3}
	Let $i$ be an ordinary vertex and let $t \in \N$.
	All children of $i$ have the same state at time $t$.
\end{lemma}
\begin{proof}
	This follows directly by a symmetry argument. For
	every pair of children $j_1$ and $j_2$ of $i$
	there is an automorphism of the graph $\mathcal{G}$
	that exchanges $j_1$ and $j_2$.
	The initial state and the functions
	defining the dynamics are invariant under such
	automorphisms of the graph, hence the same
	must hold for every state in the trajectory.
\end{proof}

\begin{lemma}
\label{lem:claim4}
	Let $i$ be an ordinary vertex.
	If $X_i(t)=0$, then $X_j(t+1)=0$ for
	all children $j$ of $i$.
\end{lemma}
\begin{proof}
	Based on Lemma \ref{lem:claim3} we have to differentiate between only three cases.
	In the first case all children of $i$ are cooperators.
	By \cref{lem:claim2} they will switch to defection.
	In the second case they are boundary defectors.
	Therefore all of their children must be cooperators
	and again \cref{lem:claim2} shows that they will switch to
	defection. In the last case the children are inner defectors
	which can not change their strategy.
\end{proof}

The dynamics along the special vertices is very simple to describe.
Let $f: \{0, \ldots, 2q-6 \} \to \N$ be the function given by $f(t):= |t-q+3|+1$,
see \cref{fig:function-graph}.
\begin{figure}
	\begin{center}
		\includegraphics{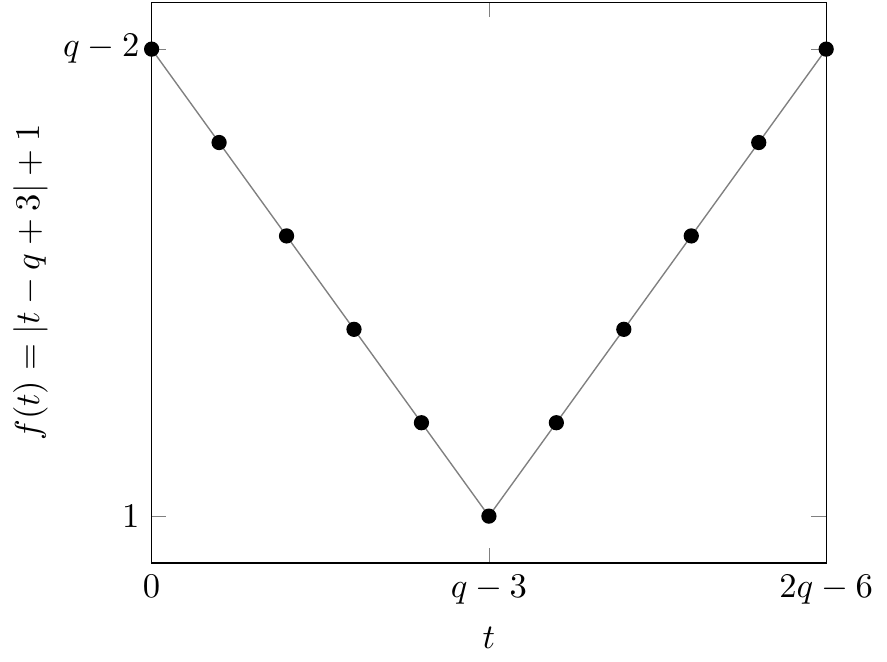}
	\end{center}
	\caption{The function $f$ governing the shrinking and expansion of
		cooperation among the special vertices for $q=8$.}
	\label{fig:function-graph}
\end{figure}
A special vertex $h_\ell$ is cooperating at time $t$ if
and only if $\ell \leq f(t)$.
This is shown together with a description of the dynamics of the
strategies of the ordinary vertices in the following theorem.
Notice that property \eqref{item:growing-phase-cooperators}
and property \eqref{item:growing-phase-defectors} in Theorem \ref{tre-thm:02} immediately
imply that $\mathcal{X}$ has period $2q-6$
and property \eqref{item:strategy-special}
implies that $\mathcal{X}$ has no shorter period.

\begin{theorem}
	\label{tre-thm:02}
	The following invariants hold for the dynamics when $0 \leq t \leq 2q-6$
	\begin{enumerate}[(a)]
		\item
		\begin{align*}
		X_{h_\ell}(t)=\begin{cases}
		1 & \text{if }\ell \leq f(t) \\
		0 & \text{otherwise}
		\end{cases}, && \text{for all } h_\ell \in H_\ell.
		\end{align*}
    \label{item:strategy-special}
		\item $h_\ell$ is an inner cooperator if and only if $\ell <
		f(t)$. \label{item:special-IC}
		\item $h_\ell$ is an inner defector if and only if $\ell > f(t)+1$.
		\label{item:special-ID}
	\end{enumerate}
	In the shrinking phase $(0\leq t < q-3)$ additionally the following
	properties hold.
	\begin{enumerate}[(a)]
		\setcounter{enumi}{3}
		\item For $m \leq f(t)+1$ and $i \in I_m \cap N_1(h_m)$ we
		have $X_i(t)=1$. \label{item:shrinking-phase-cooperators}
		\item For $m > f(t)+1$ and $i \in I_m \cap N_{\leq m-f(t)-1}(h_m)$ we
		have $X_i(t)=0$. \label{item:shrining-phase-defectors}
	\end{enumerate}
	In the expanding phase $(q-3 \leq t \leq 2q-6)$, we have
	\begin{enumerate}[(a)]
		\setcounter{enumi}{5}
		\item All vertices at level at most $f(t)$ are
                  cooperating. \label{item:growing-phase-cooperators}
		\item All vertices at level $n$ with $f(t) < n \leq
                  f(t)+3$ are defecting. \label{item:growing-phase-defectors}
	\end{enumerate}
\end{theorem}
\begin{proof}
  We show by induction that these invariants are true throughout the
  course of the dynamics.  Let $s \in \N$ and assume that the theorem
  holds for all $t \leq s$. \\

  \paragraph{Initial state; i.e.\ $s=0$:}
  Obviously, all of the points \eqref{item:strategy-special} -
  \eqref{item:shrining-phase-defectors} hold true.  Since all of the
  vertices $h_\ell$ for $\ell \leq q-3$ are inner cooperators by
  \eqref{item:special-IC}, they preserve their strategy at time $1$.
  The defecting vertex $h_{q-1}$ has utility $c$.  Thus, the
  vertex $h_{q-2}$ changes its strategy to defection at time $s+1$
  while changing the strategy of vertices in
  $I_{q-2} \cap N_{1}(h_{q-2})$ to cooperation as a consequence of
  Lemma \ref{lem:claim1}.  Every other vertex preserve its strategy at
  time $s=0$ and thus, the points \eqref{item:strategy-special} -
  \eqref{item:shrining-phase-defectors}
  hold true at time $1$. \\

  \paragraph{Shrinking phase; i.e.\ $0 < s <q-3$:}
  The vertex $h_{f(s)+1}$ is defecting and has one defecting neighbor
  $h_{f(s)+2}$ by \eqref{item:strategy-special}.  The children of
  $h_{f(s)+1}$ are cooperating by
  \eqref{item:shrinking-phase-cooperators}.  Thus, using Lemma
  \ref{lem:claim2}, the vertex $h_{f(s)+1}$ and all of his neighbors
  are defecting in the next time step.  Together with
  \eqref{item:special-ID}, this proves the point
  \eqref{item:strategy-special} for time $s+1$.  Using
  \eqref{item:shrinking-phase-cooperators}, this also immediately
  implies \eqref{item:special-IC} (the boundary cooperators closest to
  the root $h_0$ of the cluster containing $h_0$ are at level
  $q-2-s$).

  The vertices $h_\ell$ for $\ell \geq f(s)+1$ are inner defectors by
  \eqref{item:special-ID}.  Moreover, their children are all defecting
  by \eqref{item:shrining-phase-defectors}.  Thus, $h_\ell$ stay inner
  defectors for $\ell \geq f(s)+1$.  The vertex $h_{f(s)}$ is a
  boundary defector by \eqref{item:special-IC} and has $r$ cooperating
  neighbors (\eqref{item:shrinking-phase-cooperators} and
  \eqref{item:strategy-special}).  Lemma \ref{lem:claim2} implies the
  vertex $h_{f(s)}$ is an inner defector at time $s+1$ which is
  \eqref{item:special-ID} for the next time step.

  The invariant \eqref{item:strategy-special} implies that the
  predecessors of all vertices in $I_m \cap N_1(h_m)$ are cooperating
  for $m \leq f(s)+1$. The children of a specific vertex $v$ in
  $I_m \cap N_1(h_m)$ are either all defecting (\cref{lem:claim3})
  and then \cref{lem:claim1} ensures the preservation
  of cooperation in $s+1$. If the children of $v$ are cooperating then
  the
  vertex $v$ is an inner cooperator and preserves its strategy.

  As a trivial consequence of \cref{lem:claim4}, if all vertices
  in $I_m \cap N_{\leq m-f(s)-1}(h_m)$ are defecting for $m > f(s)+1$
  then all vertices in $I_m \cap N_{\leq m-f(s)}(h_m)$ are defecting
  in the next time step. Furthermore, by \eqref{item:special-ID} and
  \eqref{item:shrinking-phase-cooperators} we can apply
  \cref{lem:claim1} to the vertex $h_{q+s-2}$. This gives
  \eqref{item:shrining-phase-defectors}.\\

  \paragraph{Phase switch; i.e.\ $s=q-2$:}
  We already established
  \eqref{item:strategy-special} - \eqref{item:shrining-phase-defectors}
  at time $s+1$.
  We still have to show, that \eqref{item:growing-phase-cooperators}
  and \eqref{item:growing-phase-defectors} hold at time $s+1$.
  We have $f(q-3)=1$. There are no ordinary vertices at level one, hence \eqref{item:shrinking-phase-cooperators}
  holds at time $s+1$ by \eqref{item:strategy-special}.
  This also shows \eqref{item:growing-phase-defectors} for special
  vertices. There is also no ordinary vertex at level two and three, hence we
  only have to show \eqref{item:growing-phase-defectors} for ordinary
  vertices at level four. They are contained in $N_{\leq m-2} \cap
  I_m$ for some $m=3$, hence they are defecting at time $s+1$ by \eqref{item:shrining-phase-defectors}. \\

  \paragraph{Growing phase; i.e.\ $q-3 \leq s <2q-6$:}
  \cref{lem:claim1} together with
  \eqref{item:growing-phase-cooperators} and
  \eqref{item:growing-phase-defectors} implies that all vertices at
  level at most $f(s)+1$ will cooperate at time $s+1$, hence
  \eqref{item:growing-phase-cooperators} holds.  This also implies
  \eqref{item:special-IC}.

  The special vertices $h_{f(s)+2}, \dots, h_{q-1}$ are inner
  defectors by \eqref{item:special-ID} and hence also defect at time
  $s+1$.  Therefore \eqref{item:strategy-special} is satisfied.  If
  $f(s)+3<q$, property \eqref{item:growing-phase-defectors}
  automatically holds at time $s+1$.  Consider $s$ with $f(s)+3<q$.
  An ordinary vertex at level $f(s)+4$ is either an inner defector at
  time $s$ and hence defects at time $s+1$ or it has only cooperating
  children and hence defects by \cref{lem:claim2}.  All in all this
  shows that \eqref{item:growing-phase-defectors} is also fulfilled.
  Let $v$ be a child of a special vertex $h_{\ell}$ with
  $\ell>s+1$. By \eqref{item:special-ID} it is defecting at time $s$.
  Either it is an inner defector and hence also defects at time $s+1$
  or all its children are cooperators and it defects at time $s+1$ by
  \cref{lem:claim2}.  This established in particular that $h_{\ell}$
  is an inner defector at time $s+1$, in other words,
  \eqref{item:growing-phase-defectors}.
\end{proof}

\subsubsection{Parameter choice}
The only assumptions we needed in the dynamics section were the
inequalities \eqref{tre-eq:10}, \eqref{tre-eq:20} and the assumption
that the parameters $(a,b,c,d)$ satisfy the conditions of the HD scenario $(c>a>b>d)$.
Let such $a,b,c,d$ be given.
Clearly, $r$ can be chosen great enough such that the inequalities \eqref{tre-eq:10} and \eqref{tre-eq:20} hold.

The minimal period of the constructed trajectory is $2(q-3)$. Setting $q:= \max \{ 5, \left\lceil p_0/2 \right\rceil + 3 \}$ the period is at least $p_0$.
\qed

\begin{remark}
  In the constructions in Sections \ref{ssec:ac} and \ref{ssec:ca}, the behaviour of the number
  of cooperators or more precisely the sequence $(|\{v \in V
  \setsep X_v(t) = 1 \} | )_{t \in \N_0}$ was rather boring. During one
  period of the trajectory it was growing and reset to the initial
  value at the end of the period. The behaviour of this sequence
  is much more interesting for our tree construction as shown in \cref{fig:card}.
\end{remark}

\begin{figure}[htbp]
\includegraphics[width=0.45\textwidth]{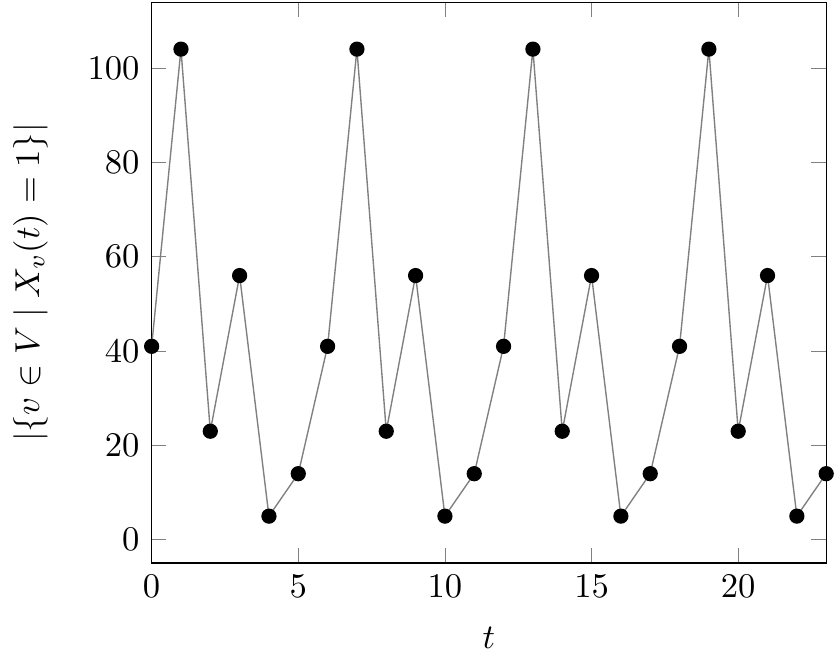}
\hspace{0.06\textwidth}
\includegraphics[width=0.45\textwidth]{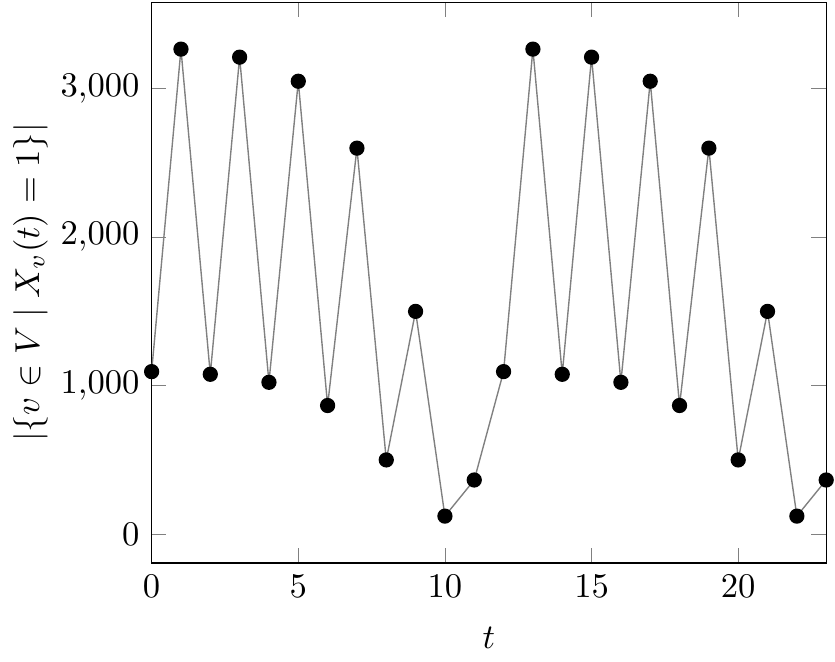}
\caption{Development of the number of cooperators for
  the evolutionary game on the tree $\mathcal{G}$ in Section \ref{sec:tre} with $r=3$ and
  game theoretic parameters $(a,b,c,d)=(1,0.7,2,0)$.
 On the left the tree has depth $q=6$, on the right $q=9$.}
\label{fig:card}
\end{figure}

\subsection{Example}
\label{tre-sec:example}
Figures \ref{tre-fig:ex0} -- \ref{tre-fig:ex5} depict an example of a
trajectory on an evolutionary game on a graph constructed in Section
\ref{tre-sec:proof}. Cooperators are depicted with black circles, defectors are depicted with white ones.
The players changing strategy in the current time step are highlighted with a dashed circle.
The parameters of this graph are $r=3, q=6$.  This trajectory can be
observed for example for parameter vector $(a,b,c,d) = (1,0.6,2,0)$
satisfying the inequalities \eqref{tre-eq:10}, \eqref{tre-eq:20}. Note
that the inequality
\begin{align}
\label{eq:tre-eq:30}
b > \frac{c+rd}{r+1}
\end{align}
holds for such a choice of parameters. Cooperation then spreads from
outer cooperators towards the leaves between $X(2)$ and $X(3)$.
In contrary, for $b \in (0,0.5)$ the inequality \eqref{eq:tre-eq:30}
does not hold anymore.  The outer cooperators (cooperators not in the
cluster containing the root $h_0$) then vanish in $X(3)$ and they do
not spread cooperation further.  The strategy vectors $X(t)$ and
$X(t+6)$ coincide for $t \in \mathbb{N}_0$. 

This example and an example of an evolutionary game on a graph with $q=7$ and all other parameters remaining the same can be found online in \cite{Epperlein2017}.

\begin{figure}[p]
\centering
\includegraphics[height=.43\textheight]{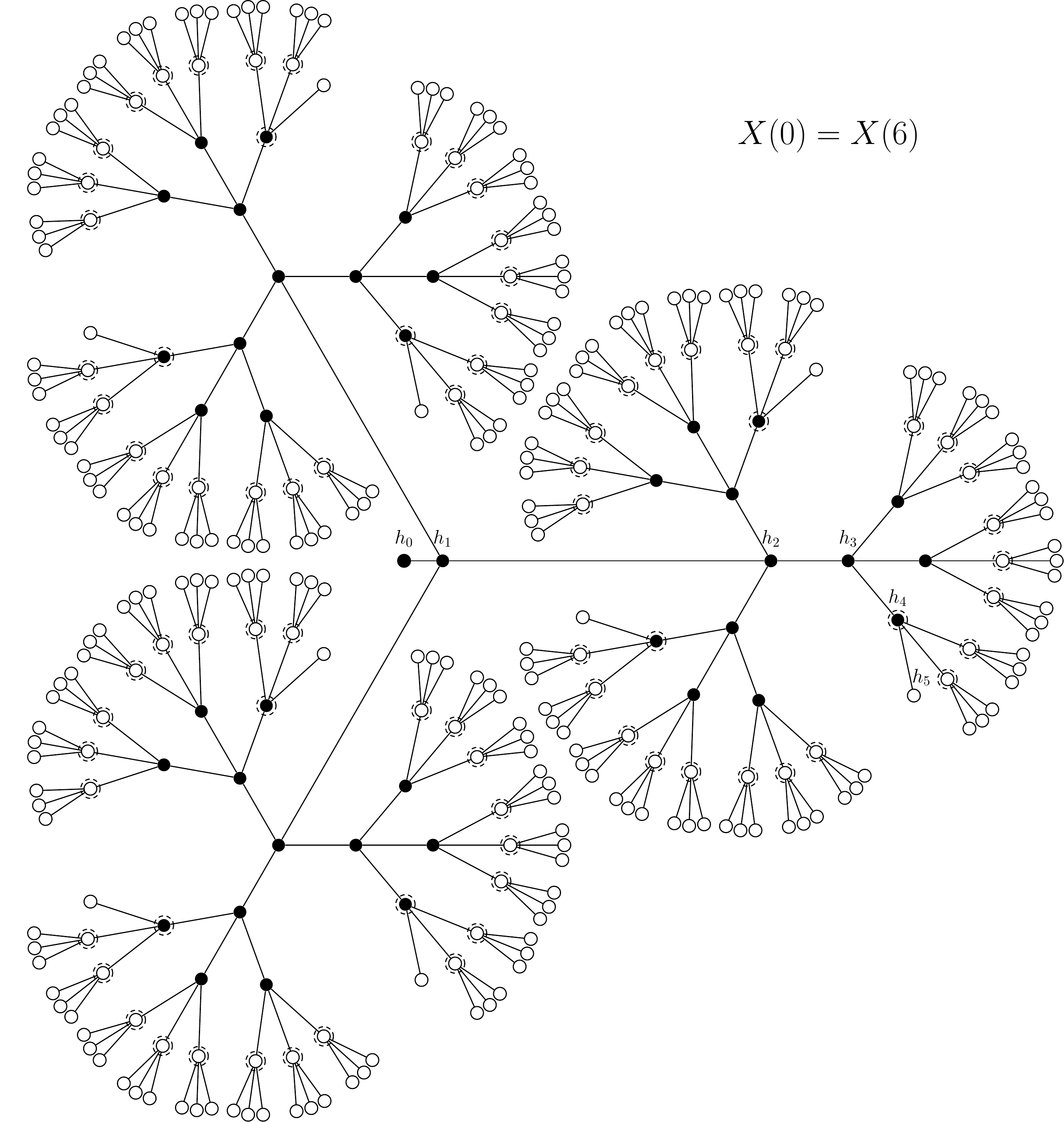}
\caption{The example from Section \ref{tre-sec:example}.}
\label{tre-fig:ex0}
\end{figure}
\begin{figure}[p]
\centering
\includegraphics[height=.43\textheight]{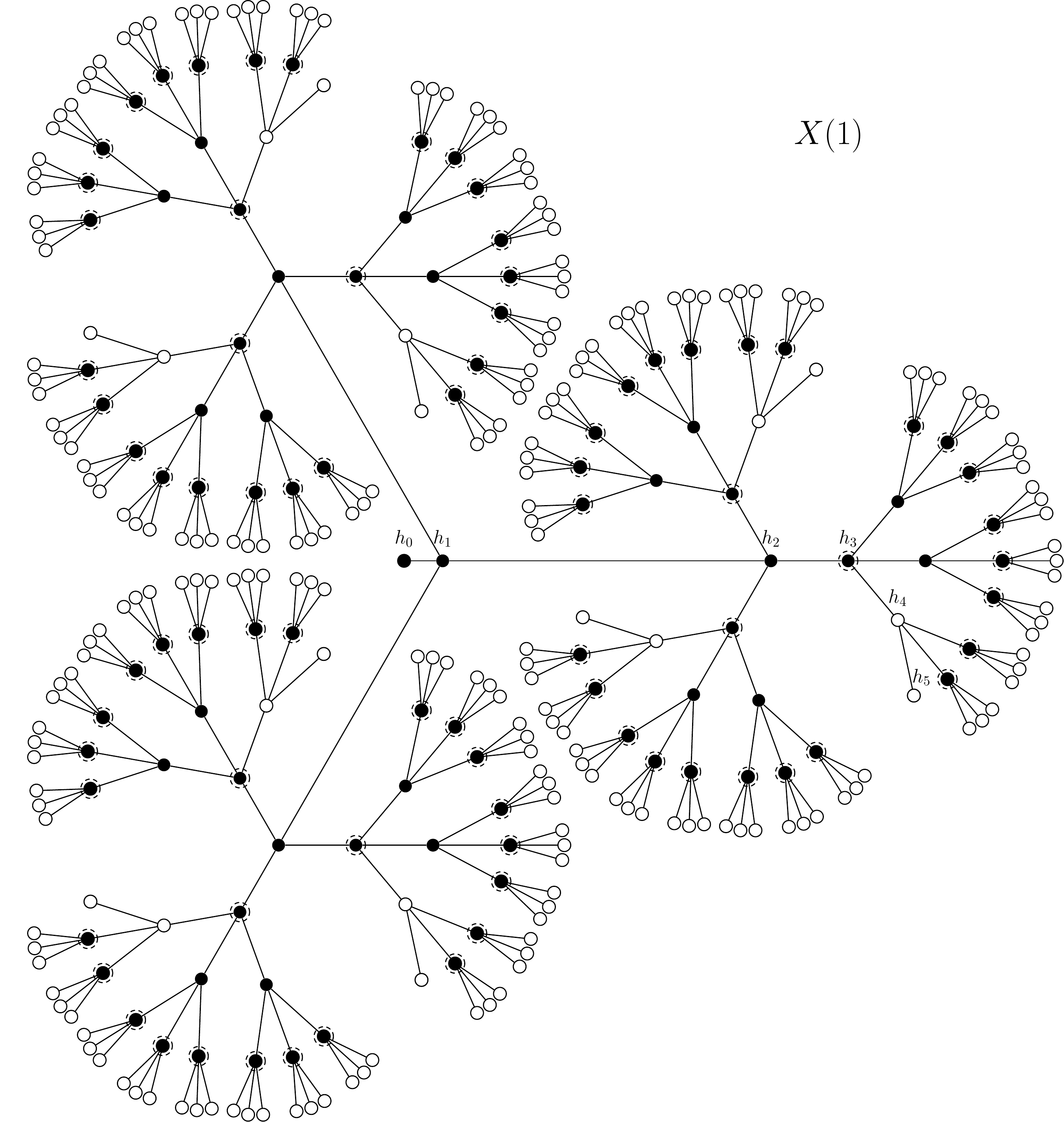}
\caption{The example from Section \ref{tre-sec:example}.}
\label{tre-fig:ex1}
\end{figure}
\begin{figure}[p]
\centering
\includegraphics[height=.43\textheight]{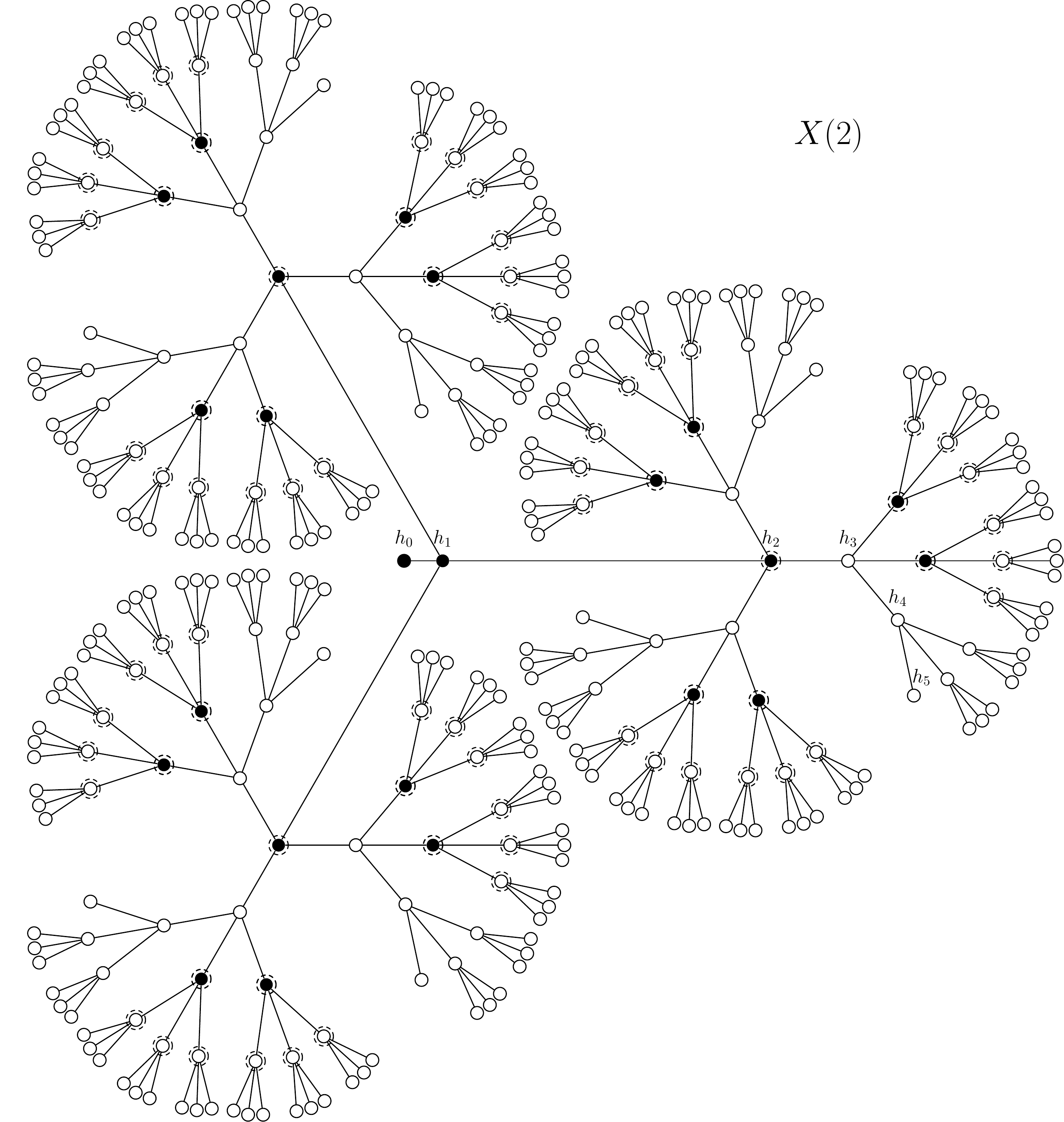}
\caption{The example from Section \ref{tre-sec:example}.}
\label{tre-fig:ex2}
\end{figure}
\begin{figure}[p]
\centering
\includegraphics[height=.43\textheight]{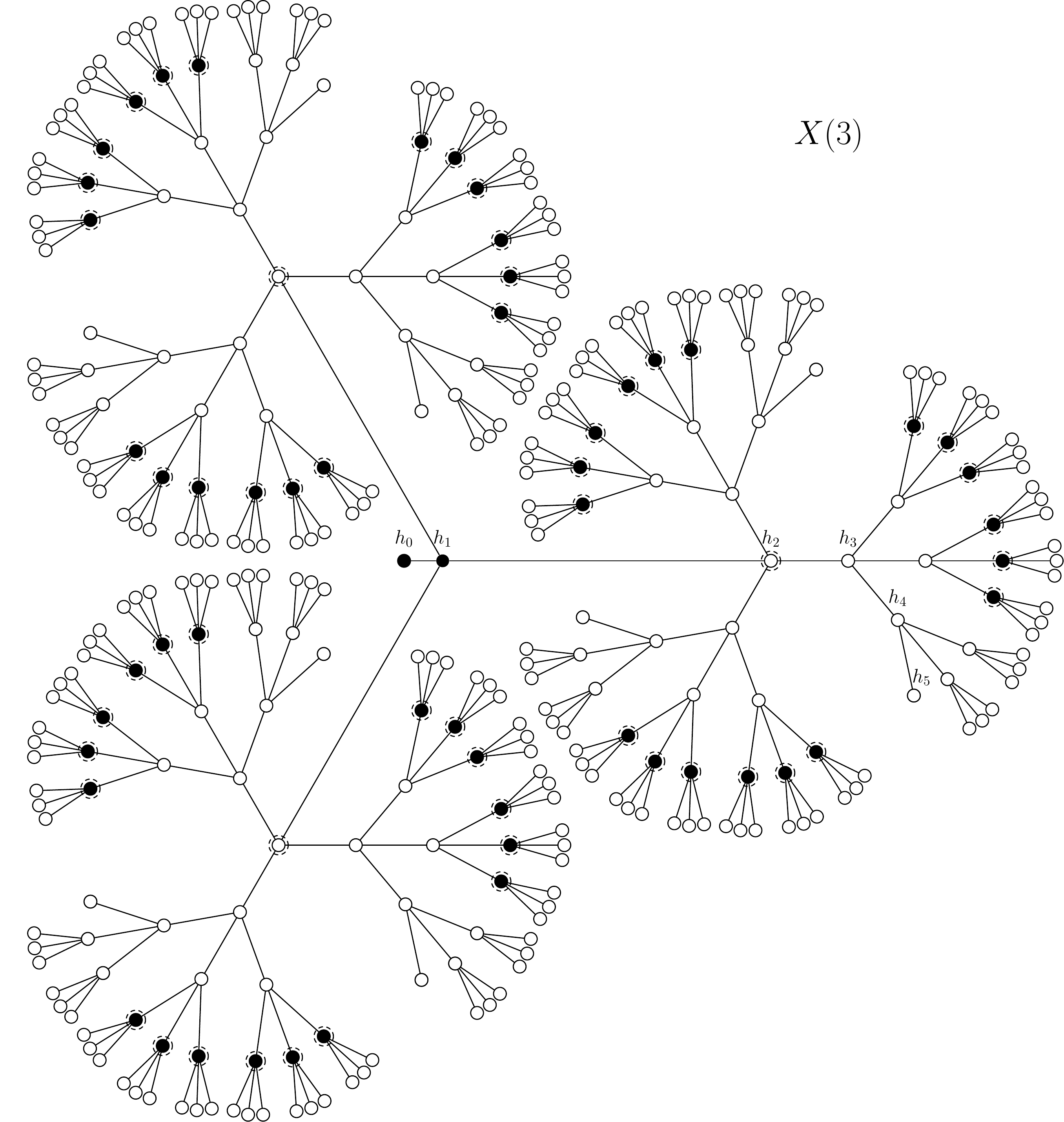}
\caption{The example from Section \ref{tre-sec:example}.}
\label{tre-fig:ex3}
\end{figure}
\begin{figure}[p]
\centering
\includegraphics[height=.43\textheight]{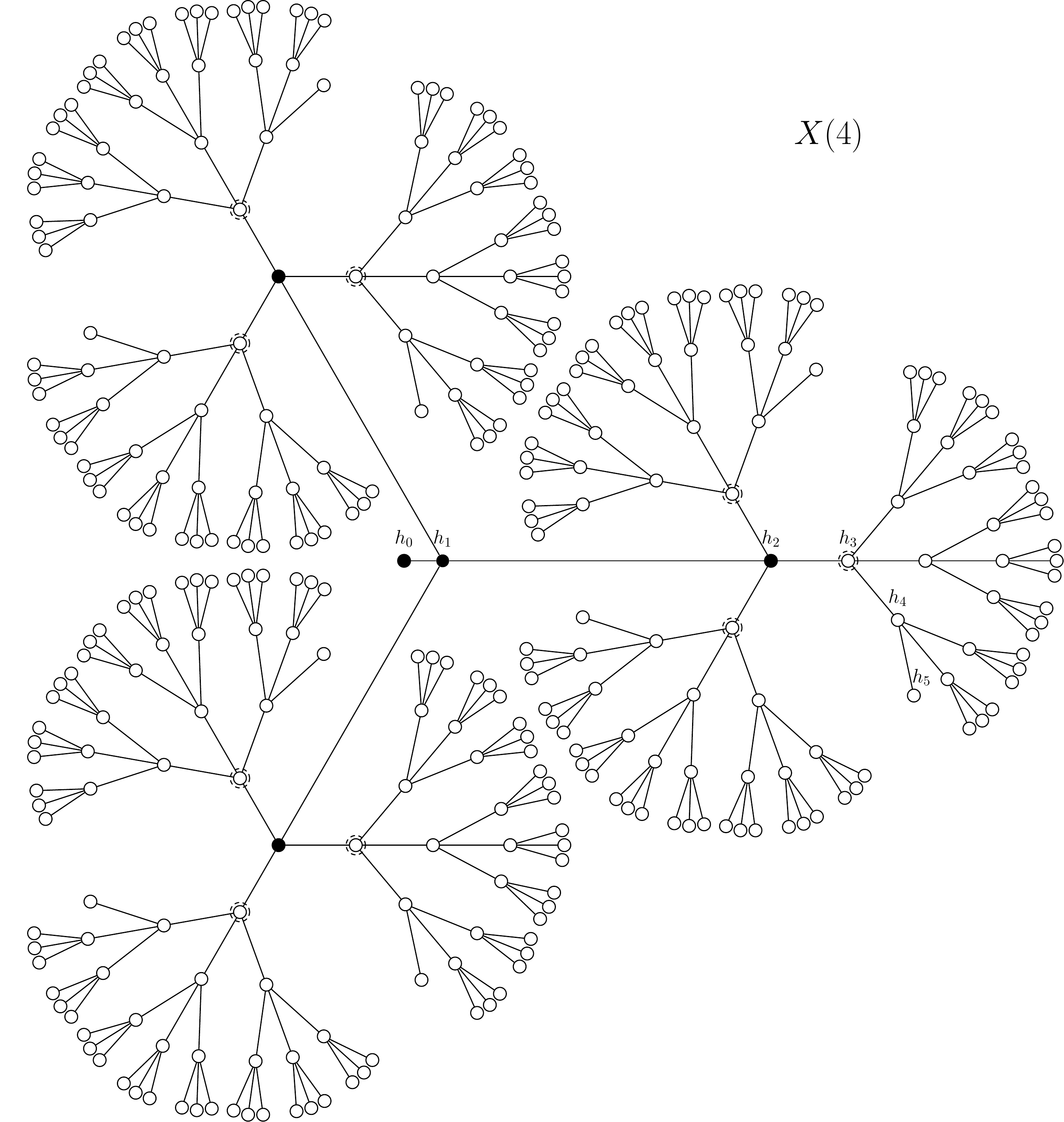}
\caption{The example from Section \ref{tre-sec:example}.}
\label{tre-fig:ex4}
\end{figure}
\begin{figure}[p]
\centering
\includegraphics[height=.43\textheight]{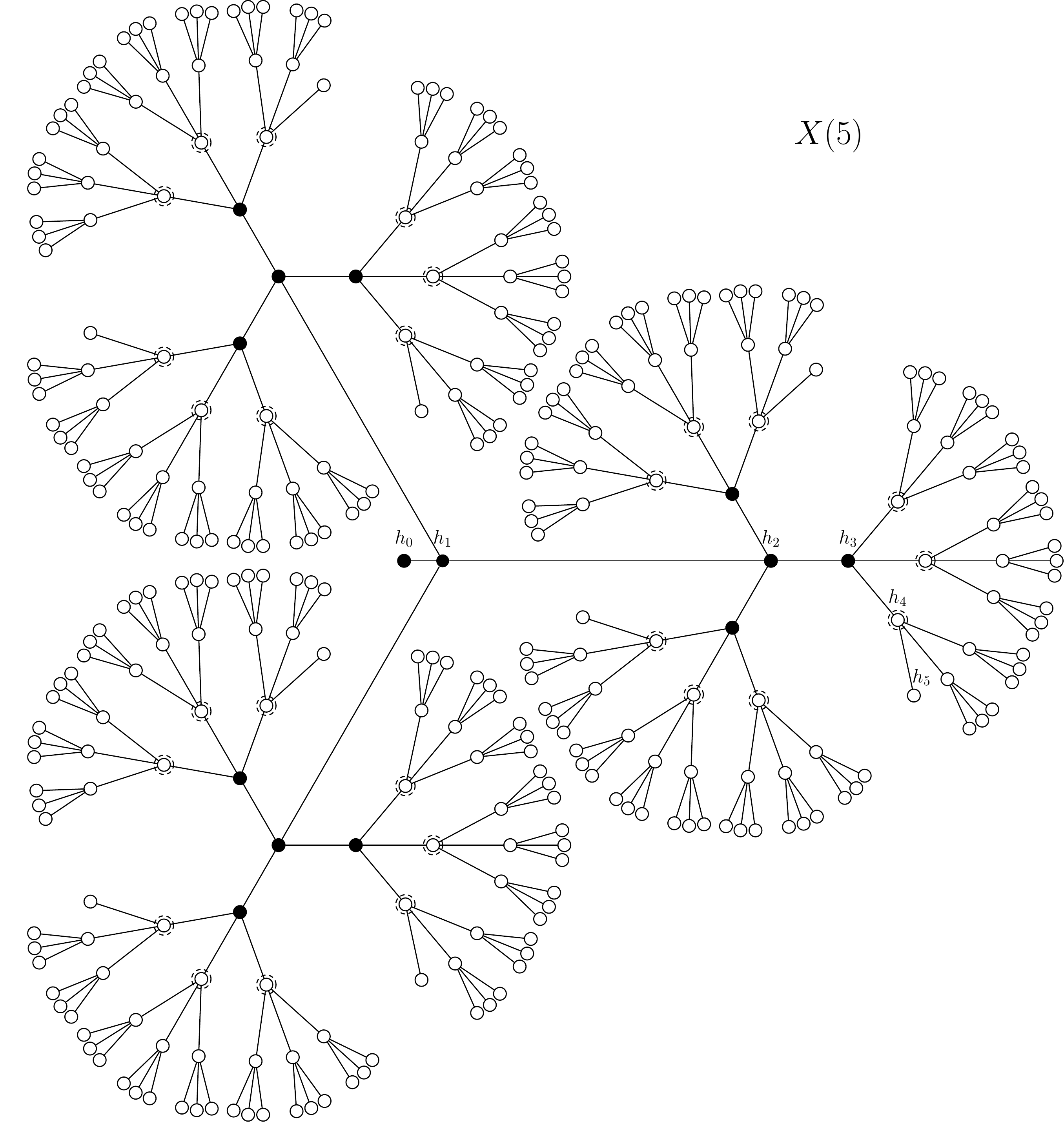}
\caption{The example from Section \ref{tre-sec:example}.}
\label{tre-fig:ex5}
\end{figure}
 
\section{Conclusion}
\label{sec:ccl}
We showed that on arbitrary graphs the game theoretic parameters
can not exclude periodic behavior with long periods.
Our proofs hold also true for a small perturbation of the game-theoretical parameters $a,b,c,d$ as a consequence of the generic payoff assumption.

Our constructions rely heavily on the fact that we can choose the graph
parameters arbitrarily.
This no longer works if we restrict to certain classes of graphs.
For example Theorem \ref{tre-thm:1} partially answers Question \ref{pre-que:1} while restricting to the parameters $(a,b,c,d)$ satisfying the conditions of the HD scenario, $c>a>b>d$, and the class of acyclic graphs.

Natural classes of graphs we might restrict ourselves to are $k$-regular graphs (every
vertex has exactly $k$ neighbors), vertex-transitive graphs (every
pair of vertices can be exchanged by a graph automorphism) or planar graphs (the
graph can be drawn in the plane without edge crossings).
This leads for example to the following question.
\begin{question}
  For which game theoretic parameters $(a,b,c,d)$ and positive integers $k, p$ is there
  a $k$-regular graph $\mathcal{G}$ such that the corresponding evolutionary
  game with synchronous update and imitation dynamics on $\mathcal{G}$ has a periodic trajectory with minimal period $p$?
\end{question}

\section*{Acknowledgments} The first author would like to acknowledge the
project LO1506 of the Czech Ministry of Education, Youth and Sports for
supporting his visit at the research centre NTIS – New Technologies for
the Information Society of the Faculty of Applied Sciences, University
of West Bohemia. The second author was supported by
the Grant Agency of the Czech Republic Project No. 15-07690S.

\providecommand{\href}[2]{#2}
\providecommand{\arxiv}[1]{\href{http://arxiv.org/abs/#1}{arXiv:#1}}
\providecommand{\url}[1]{\texttt{#1}}
\providecommand{\urlprefix}{URL }

\end{document}